\documentclass[12pt,a4paper]{amsart} 
\usepackage{geometry}
\usepackage{tabls}
\usepackage{url}
\usepackage[utf8]{inputenc}
\usepackage{amsfonts}
\usepackage{amssymb}
\usepackage{mathrsfs}
\usepackage{amsmath}
\usepackage{amscd}
\usepackage{fancyhdr}
\usepackage{enumerate}
\usepackage{paralist}
\usepackage{xypic}
\usepackage{graphicx}
\usepackage{hyperref}
\usepackage{cite}
\usepackage{lipsum}
\usepackage{footmisc}
\usepackage[all,cmtip]{xy}

\allowdisplaybreaks

\numberwithin{equation}{section}
\theoremstyle{plain}
\newtheorem{theorem}{Theorem}[section]
\newtheorem{lemma}[theorem]{Lemma}
\newtheorem{remark}[theorem]{Remark}

\numberwithin{equation}{section}

\def\a{\alpha}

\def\L{\Lambda}
\def\n{\,\vert\,}

\def\tr{{\rm tr\/}}

\def\rank{{\rm rank\/}}

\def\ind{{\rm ind\/}}
\def\Stab{{\rm Stab\/}}

\def\R{\mathbb{R} }

\def\N{\mathbb{N}}

\def\fg{\mathfrak{g}}

\begin{document}
	
	\title{A variation of the reduced  Vafa-Witten equations on 4-manifolds}

	
	\author{Ren Guan}
	\address{School of Mathematics and Statistics, 
		Jiangsu Normal University, Xuzhou 221100, China}
	\email{guanren@jsnu.edu.cn}

	\begin{abstract}
		In this paper we consider a variation of the Vafa-Witten equations on compact, oriented and smooth 4-manifolds, and construct a set of perturbation terms to establish the transversality of that equations. The new perturbed equations provide us a priori estimates of the solutions, while the original reduced Vafa-Witten equations does not. By applying the a priori estimates we show that the singular of the solutions can be removed, and then construct the Ulhenbeck closure of the moduli spaces.

	\end{abstract}
	
	
	
	\maketitle
	
	\tableofcontents
	
	\section{Introduction}
	
	Let $(X,g)$ be a compact, oriented and smooth Riemannian 4-manifold, $k\geq 3$ be an integer, $G$ a compact Lie group and denote by $\mathfrak{g}$ the Lie algebra of $G$. Let $P\to X$ a principle $G$-bundle on $X$, $\mathfrak{g}_P:=P\times_{\text{ad}}\mathfrak{g}$ the associated bundle of $P$, $\mathcal{A}_k (P)$ the space of $L_k^2$ connections on $P$, $\mathcal{G}_{k+1}(P)$ the group of $L^2_{k+1}$ gauge transformation of $P$. For a triplet $(A,B,C)\in\mathcal{A}_k (P)\times L_k^2(X,\mathfrak{su}(2)_P\otimes\Lambda^{2,+})\times L_k^2(X,\mathfrak{su}(2)_P)$, the equation
	\begin{equation}\left\{
		\begin{aligned}
			&d_A^*B +d_A C=0,\\&F_A^++\frac{1}{8}[B\centerdot B]+\frac{1}{2}[B,C]=0  \\
		\end{aligned}\right. \label{aa}
	\end{equation}
	is called \emph{Vafa-Witten equation}, which is invited by Vafa and Witten when studying an $N=4$ topologically twisted supersymmetric Yang-Mills theory on 4-manifolds\cite{VW}. Equation \eqref{aa} is constructed to compute the Euler characteristic of instanton moduli space, and the construction also makes \eqref{aa} satisfy some "vanishing theorem", i.e., for some special 4-manifolds $X$, \eqref{aa} has the same solution as the ASD equation $F_A^+=0$\cite[Theorem 2.1.3]{Ma}.
	
	On closed 4-manifolds, Vafa-Witten equation is equivalent to
	\begin{equation}
		\left\{
		\begin{aligned}
			&d_AC=d_A^*B=0\\&F_A^++\frac{1}{8}[B\centerdot B]=[B,C]=0.\\
		\end{aligned}\right.
		\label{ab}
	\end{equation}
	When $G=SU(2)$ or $SO(3)$ and the connection $A$ is irreducible, $d_AC=0$ implies $C=0$ (cf.  \cite[\S 2.1]{Ma}), and then the equation \eqref{ab} can be further reduced  to
	\begin{equation}\left\{
		\begin{aligned}
			&d_A^*B=0\\&F_A^++\frac{1}{8}[B\centerdot B]=0.\\
		\end{aligned}\right.
		\label{ac}
	\end{equation}
	We call equation \eqref{ac}  the \emph{reduced Vafa-Witten equation}.
	
	Unlike $d_A^*\oplus d_A:L_k^2(X,\mathfrak{su}(2)_P\otimes\Lambda^{2,+})\times L_k^2(X,\mathfrak{su}(2)_P)\to L_k^2(X,\mathfrak{su}(2)_P\otimes\Lambda^{1})$, the operator $d_A^*:L_k^2(X,\mathfrak{su}(2)_P\otimes\Lambda^{2,+})\to L_k^2(X,\mathfrak{su}(2)_P\otimes\Lambda^{1})$ is not elliptic, and  also not a Fredholm operator, the cancellation of $C$ in \eqref{aa} destroys the finiteness of the expected dimension of the moduli space of solutions to \eqref{aa}. To settle this down,  we replace $d_A^*B=0$ by the following one:
	$$d_A^+d_A^*B=0.$$
	
	If $d_A^+d_A^*B=0$, take inner product with $B$ we have
	$$0=\langle d_A^+d_A^*B,B\rangle_{L^2}=\langle d_A^*B,d_A^*B\rangle_{L^2}=\vert \vert d_A^*B\vert \vert _{L^2}^2,$$
	so we must have $d_A^*B=0$, which means
	\begin{equation}\left\{
		\begin{aligned}
			&d_A^+d_A^*B=0\\&F_A^++\frac{1}{8}[B\centerdot B]=0\\
		\end{aligned}\right.
		\label{ad}
	\end{equation}
	and \eqref{ac} have the same solutions, and  $d_A^+d_A^*:\Omega^{2,+}(X)\to\Omega^{2,+}(X)$ is self-adjoint, hence a Fredholm operator. So once the transversality is established, the expected dimension of the moduli space of \eqref{ad} is finite (See \cite{D1,D2} for the discussion of ASD equations, \cite{Fe,FL} for $PU(2)$-monopole equations and \cite{Mor} for Seiberg-Witten equations). We call \eqref{ad}  the \emph{variated reduced Vafa-Witten equation}, or \emph{VRVW equation} for short, which is the main object studied in this paper.
	
	Recall that both of Donaldson's polynomial invariants and Seiberg-Witten invariants can distinguish the different differential structures on homeomorphic 4-manifolds. For example, one can use both invariants to show that $K3\#\overline{\mathbb{CP}^2}$ and $\#3\mathbb{CP}^2\#20\overline{\mathbb{CP}^2}$ are homeomorphic but not diffeomorphic (see \cite[Theorem A]{D1}, \cite[\S 2, \S 4]{W}). The two invariants are both defined via the moduli space of solutions to certain equations derived from gauge theory. Donaldson polynomial invariants are defined via the $SU(2)$ anti-self-dual (ASD) Yang-Mills equation on 4-manifolds with non-abelian structure group \cite{D1,D2},  and Seiberg-Witten invariants rely on the $U(1)$ monopole equations \cite{Mor,W}.
	
	Before defining the invariants, Donaldson, Seiberg and Witten construct perturbations to the corresponding equations first (see \cite{D1,Mor} for details), the main reason for doing this is that the moduli spaces of solutions to the original equations don't form smooth manifolds due to the lack of transversality. After constructing proper perturbations, the moduli spaces of solutions to the perturbed equations form smooth manifolds and after a lot of sophisticated analysis, the invariants are constructed and can be applied to distinguish the differential structures of 4-manifolds who are topological homeomorphism. 
	
	Besides the transversalty, the compactness of the moduli spaces is also of great importance. For Seiberg-Witten equations, there is \emph{a priori} estimates for the solutions, which means the bounds don't rely on the Spin$^c$ structures but only on the geometry of the base manifolds $X$\cite[Corollary 5.2.2]{Mor}. The reason why we can deduce this estimates is that the quadratic form $q(\psi)$ of Seiberg-Witten equations is positive definite, i.e., there is a positive number $C>0$ such that $q(\psi)\geq C\n\psi\n^2$. But for Vafa-Witten equations, the method doesn't work, because the quadratic form $[B\centerdot B]$ is not positive definite, so there's no longer \emph{a priori} estimates for Vafa-Witten equations.
	
    Previously,  to derive the transversality result, Tanaka constructs a perturbation of Vafa-Witten equation on closed symplectic 4-manifolds which contains five parameters, then show that for generic perturbations, the moduli spaces are zero-dimensional manifolds\cite{Tan}. And recently, Bo Dai and I show that the transversality is satisfied at the full rank  part of the Vafa-Witten moduli spaces, which implies that part of moduli spaces are smooth manifolds of dimension zero. In this paper, we partially solve the problem of \emph{a priori} estimates of Vafa-Witten equations on a general compact, oriented and smooth Riemannian 4-manifolds by constructing a perturbation to the equation \eqref{ad}. After constructing the perturbation, we first establish the transversality of the perturbed VRVW equations, then we show that the dimensions of the non-asd part of moduli spaces of the perturbed VRVW equations are finite for a generic choice of the perturbation terms. Also the perturbation will provide us a priori estimates of $B$ and $d_AB$, which is new for Vafa-Witten equations. And then on the basis of the estimates, we construct the Ulhenbeck closure of the moduli spaces of the perturbed VRVW equations.
	
	Moreover, when $X=B^4$, the unit ball of $\mathbb{R}^4$ and $P= X\times G$, the product bundle, then we show that the moduli space $\mathcal{M}_{PVRVW,\kappa(P)}^*(t,\tau)$ is compact for $t$ sufficiently large. Our results are summerized as the following theorems.
	\begin{theorem}\label{m1}Let $(X,g)$ be a compact, oriented and smooth Riemannian 4-manifold, $G=SU(2)$ or $SO(3)$, $P\to X$ a principle $G$-bundle, $\mathcal{T}^r:=\R^+\times C^r(GL(\Lambda^{2,+}))$ where $r>k$, $\kappa(P)$ the characteristic number of $P$. If $\kappa(P)\geq 3/8(1-b_1(X)+b_2^+(X))$, then there is a first-category subset $\mathcal{T}^r_{fc}\subset\mathcal{T}^r$ such that for all $(t,\tau)$ in $\mathcal{T}^r-\mathcal{T}^r_{fc}$, after module the group of gauge transformations $\mathcal{G}_{k+1}(P)$, the solutions $(A,B)\in\mathcal{A}_k (P)\times L_k^2(X,\mathfrak{su}(2)_P\otimes\Lambda^{2,+})$ with $B\not\equiv 0$ of the equation
	\begin{equation}\label{yz}\left\{
			\begin{aligned}
				&d_A^+d_A^*B+t\langle B,B\rangle^2B+\tau[[B\centerdot B]\centerdot[B\centerdot B]]=0\\
				&F_A^++\frac{1}{8}[B\centerdot B]=0\\
	\end{aligned}\right.\end{equation}
	constitute a smooth manifold $\mathcal{M}_{PVRVW,\kappa(P)}^*(t,\tau)$ of dimension $8\kappa(P)-3(1-b_1(X)+b_2^+(X))$. Moreover, we have the following a priori boundness result: there are constants $C_{t,\tau,X}$ and $K_{t,\tau,X}$ depend only on $t,\tau$ and $X$ such that for any solution $(A,B)$ of \eqref{yz}, we have $\vert\vert B\vert\vert_{L^{\infty}} \leq C_{t,\tau,X}$  and $\vert\vert d_AB\vert\vert_{L^2}\leq K_{t,\tau,X}$. For fixed $\tau$, we also have $\lim_{t\to\infty}K_{t,\tau,X}=\lim_{t\to\infty}C_{t,\tau,X}=0$.
    \end{theorem}
	\begin{theorem}\label{m2}
		Continue the above notation, denote
		$$\mathcal{M}_{PVRVW,\kappa(P)}(t,\tau):=\mathcal{M}_{PVRVW,\kappa(P)}^*(t,\tau)\cup\{A\in\mathcal{A}_k (P)|F_A^+=0\}/\mathcal{G}_{k+1}(P)$$
		i.e., the moduli space of solutions of \eqref{yz} without the restriction $B\not\equiv 0$. If the character number $\kappa(P)$ of $P$ satisfies
		$$\kappa(P)\geq-\frac{\mathrm{vol}(X)C_{t,\tau,X}^4}{48\pi^2},$$ 
		here $\mathrm{vol}(X)$ denotes the volume of the compact manifold $X$, then we can construct the Ulhenbeck closure $\overline{\mathcal{M}}_{PVRVW,\kappa(P)}(t,\tau)$ of ${\mathcal{M}}_{PVRVW,\kappa(P)}(t,\tau)$ in 
		$$\mathcal{IM}_{PVRVW,\kappa(P)}(t,\tau):=\bigcup_{l=0}^\infty\mathcal{M}_{PVRVW,\kappa(P)-l}(t,\tau)\times\mathrm{Sym}^l(X)$$
		such that  $\overline{\mathcal{M}}_{PVRVW,\kappa(P)}(t,\tau)\subset\mathcal{IM}_{PVRVW,\kappa(P)}(t,\tau)$   is  sequentially compact.
	\end{theorem}
	\begin{theorem}\label{m3}
	When $X=B^4$, $G=SU(2)$ or $SO(3)$ and $P=B^4\times G$,  then for any generic parameter $(t,\tau)\in\mathcal{T}^r_{fc}\subset\mathcal{T}^r$, there is a positive number $c_\tau$, depending only on $\tau$, such that when $t>c_\tau$, the moduli space $\mathcal{M}_{PVRVW,\kappa(P)}^*(t,\tau)$ is compact.
\end{theorem}

	\begin{remark}
		By selecting the solutions to the reduced Vafa-Witten equation \eqref{ac} with some fixed $L^2$ bounds $b\in\mathbb{R}^+$, Mares constructs the Ulhenbeck closure of the \emph{b-truncated moduli spaces}\cite[Section 3.5]{Ma}. To some extent, Theorem \ref{m2} can be regarded as a generalization of Mares' work.
	\end{remark}

	\section{Some preliminaries}

	We list some elementary definitions and notation used in this paper below. Let $V$ be a finite dimensional inner product space, $\{ e_1,\cdots, e_n\}$  an orthonormal basis for $V$ and  $\{e^1,\cdots,e^n\}$  the dual basis for $V^*$. For any $\alpha\in\Lambda^{p}V^*$, $\beta\in\Lambda^{q}V^*$, we define
	$$\alpha\centerdot\beta=(-1)^{p-1}\sum_{i=1}^n(\iota_{e_i}\alpha)\wedge(\iota_{e_i}\beta)\in\Lambda^{p+q-2}V^*,$$
	where $\iota_{e_i}$ is the contraction with $e_i$.
	
	The exterior algebra $\L^\bullet V^*$ inherits an inner product, such that $\{ e^{i_1}\wedge e^{i_2}\wedge\cdots\wedge e^{i_{p}}; 0\leq p\leq n, i_1 <i_2 <\cdots <i_p\}$ forms an orthonormal basis.  The inner product of two elements $\alpha,\beta \in \Lambda^\bullet V^*$ is denoted as $\alpha\cdot\beta$.

	Replacing $V$ by the tangent space $T_xX$ at any point $x\in X$, we can define the products $\centerdot$ and $\cdot$ for differential forms. We can equip the Lie algebra $\mathfrak{g}$ of $G$ with an invariant inner product $\langle\cdot,\cdot\rangle$, here invariance means $\langle [\xi,\eta],\zeta\rangle=\langle \xi,[\eta,\zeta]\rangle$, $\forall \xi,\eta,\zeta \in\mathfrak{g}$. If $\mathfrak{g}$ is $\mathfrak{su}(2)$, an invariant inner product is given by $\langle \xi,\eta\rangle=-\frac{1}{2}\mathrm{tr}(\xi\eta)$ \cite[(A.18)]{Ma}, where $\xi,\eta\in\mathfrak{su}(2)$ are regarded as matrices.
	
	Let $\Lambda^p(X)$ the space of real-valued $p$-forms on $X$ and $\Lambda^{2,\pm}(X)$ the space of real-valued (anti-) self-dual 2-forms on $X$ respectively. For $\fg_P$-valued forms, for example, if $\omega_1=\xi_1\otimes\alpha_1$ and $\omega_2=\xi_2\otimes\a_2$, where $\xi_1,\xi_2\in\Lambda^0(X, \fg_P)$ and $\alpha_1\in\Lambda^p(X)$, $\a_2\in\Lambda^q(X)$, we define $[\omega_1\centerdot\omega_2] =[\xi_1,\xi_2]\otimes\alpha_1\centerdot\a_2$ and $[\omega_1\cdot\omega_2]=[\xi_1,\xi_2]\otimes\alpha_1\cdot\a_2$, $\langle\omega_1\cdot \omega_2\rangle=\langle \xi_1, \xi_2\rangle \a_1\cdot \a_2$. See  \cite[Appendix]{Ma} for more detail.
	
	The \emph{rank} of a section $B\in\Lambda^{2,+}(X, \fg_P)$ is defined as follows. Choose local frames for $\fg_P$ and  $\Lambda^{2,+}(T^*X)$, then the section $B$ is represented by a $d\times 3$ matrix-valued function with respect to the local frames, where $d=\dim G$. The rank of $B$ at a point of $X$ is just the rank of the matrix at that point, and $\rank(B)$ is the maximum of the pointwise rank over $X$. If $\mathfrak{g}$ is $\mathfrak{su}(2)$, then the maximum of the rank of $B$ is $d=3$.

	Note that the Lie algebra of $SU(2)$ and $SO(3)$ are isomorphic: $\mathfrak{su}(2)\cong\mathfrak{so}(3)$, so in the following, we use $\mathfrak{su}(2)$ instead of $\mathfrak{so}(3)$ wherever $\mathfrak{so}(3)$ appears.

	\section{variated reduced Vafa-Witten equations and their perturbations}
	
	In this section we lay out the general set-up of Vafa-Witten equations and construct a perturbation for getting the transversality result. 
	
	\subsection{The VRVW map}
	
	We know that $\mathcal{A}_k (P)$ is an affine space, so for any fixed $L_k^2$ connection $A_0$, we have
	$$\mathcal{A}_k (P)=A_0+L_k^2(X,\mathfrak{su}(2)_P\otimes\Lambda^1).$$
	The \emph{configuration spaces} and \emph{target spaces} are defined by
	$$\begin{aligned}
		\mathcal{C}_k(P)&:=\mathcal{A}_k (P)\times L_k^2(X,\mathfrak{su}(2)_P\otimes\Lambda^{2,+}),\\
		\mathcal{C}_{k-1}'(P)&:=L_{k-1}^2(X,\mathfrak{su}(2)_P\otimes\Lambda^{2,+})\oplus L_{k-1}^2(X,\mathfrak{su}(2)_P\otimes\Lambda^{2,+}).
	\end{aligned}$$
	It's easy to see that $\mathcal{C}_k(P)$ is also an affine vector space, which means for any fixed $L_k^2$ pair $(A_0,B_0)\in\mathcal{C}_k(P)$, we have
	$$\mathcal{C}_k(P)=(A_0,B_0)+L_k^2(X,\mathfrak{su}(2)_P\otimes\Lambda^1)\oplus L_k^2(X,\mathfrak{su}(2)_P\otimes\Lambda^{2,+}).$$
	
	The \emph{VRVW map} is defined by
	\begin{equation}\begin{aligned}
			\mathcal{VRVW}&:\mathcal{C}_k(P)\to\mathcal{C}_{k-1}'(P),\\
			\mathcal{VRVW}(A,B)&:=\begin{pmatrix}
				d_A^+d_A^*B\\
				F_A^++\frac{1}{8}[B\centerdot B]\\
			\end{pmatrix}.
		\end{aligned} \label{cc}
	\end{equation}
	The action of an $L_{k+1}^2$ gauge transformation $\zeta\in\mathcal{G}_{k+1}(P)$ on any $(A,B)\in\mathcal{C}_k(P)$ is given by
	$$\zeta\cdot(A,B):=(A-(d_A\zeta^{-1})\zeta,\zeta^{-1}B\zeta),$$
	and a connection $A$ is called \emph{irreducible} if $\Stab(A):=\{\zeta\in\mathcal{G}_{k+1}(P):\zeta\cdot A=A\}=Z(G)$, the center of $G$. 
	
	It's not hard to see that
	$$\mathcal{VRVW}\left(\zeta\cdot(A,B)\right)=\begin{pmatrix}
		\zeta^{-1}\left(d_A^+d_A^*B\right)\zeta\\
		\zeta^{-1}\left(F_A^++\frac{1}{8}[B\centerdot B]\right)\zeta\\
	\end{pmatrix},$$
	so the map $\mathcal{VRVW}$ is gauge-equivariant (cf.  \cite[\S 3.2.1]{Ma}). Define the quotient space $\mathcal{B}_k (P):=\mathcal{C}_k (P)/\mathcal{G}_{k+1}(P)$ (cf. \cite[Propostion 2.8]{FL},  \cite[Theorem 3.2.3]{Ma}).
	
	Denote by $\mathcal{C}_k^\diamond (P) \subset\mathcal{C}_k(P)$ the pairs $(A,B)$ with $A$ irreducible and $B\not\equiv 0$. We call it the \emph{non-asd} part of $\mathcal{C}_k(P)$ (When $B\equiv 0$, equation \eqref{ad} becomes ASD equation $F_A^+=0$, that's why when call the part of solutions where $B\not\equiv 0$ the non-asd part). The slice theorem (cf.  \cite[Propostion 2.8]{FL} \cite[Theorem 3.2.3]{Ma}) implies that the quotient space $\mathcal{B}_k^\diamond (P):=\mathcal{C}_k^\diamond (P)/\mathcal{G}_{k+1}(P)\subset\mathcal{B}_k(P)$ is an open and smooth Hilbert submanifold of $\mathcal{B}_k(P)$. 
	
	\subsection{The perturbed  VRVW map}
	
	We now construct the \emph{perturbed VRVW map} as follows:
	\begin{equation}\begin{aligned}
			\mathcal{PVRVW}&:\mathcal{T}^r\times\mathcal{C}_k(P)\to\mathcal{C}_{k-1}'(P),\\
			\mathcal{PVRVW}(t,\tau,A,B)&:=\begin{pmatrix}
				d_A^+d_A^*B+t\langle B,B\rangle^2B+\tau[[B\centerdot B]\centerdot[B\centerdot B]]\\
				F_A^++\frac{1}{8}[B\centerdot B]
			\end{pmatrix}, \label{cd}
		\end{aligned}
	\end{equation}
	where $\mathcal{T}^r:=\R^+\times C^r(GL(\Lambda^{2,+}))$ denotes the Banach manifold of $C^r$ perturbation parameters $\tau$ (with $r$ large enough, say $r>k$). The gauge group $\mathcal{G}_{k+1}(P)$ acts trivially on the space of perturbations $\mathcal{T}^r$, so the map $\mathcal{PVRVW}$ is also gauge-equivariant, and $\mathcal{PM}_k(P):=\mathcal{PVRVW}^{-1}(0)/\mathcal{G}_{k+1}(P)\subset\mathcal{T}^r\times\mathcal{B}_k(P)$ is the parametrized moduli space of the perturbed VRVW equation. Let $\mathcal{PM}_k^\diamond(P)=\mathcal{PM}_k(P)\cap(\mathcal{T}^r\times\mathcal{B}_k^\diamond(P))$. Note that the gauge-equivariant map $\mathcal{PVRVW}$ defines a section of a Banach vector bundle $\overline{\mathcal{E}}_k$ over $\mathcal{T}^r\times\mathcal{B}_k^\diamond(P)$ with total space $\overline{\mathcal{E}}_k:=(\mathcal{T}^r\times\mathcal{C}_k^\diamond(P))\times_{\mathcal{G}_{k+1}(P)}\mathcal{C}'_{k-1}(P)$. In particular, the parametrized moduli space $\mathcal{PM}_k^\diamond(P)$ is the zero set of the section $\mathcal{PVRVW}(\cdot,\cdot)$ of the vector bundle $\overline{\mathcal{E}}_k$ over $\mathcal{T}^r\times\mathcal{B}_k^\diamond(P)$. 
	
	\subsection{The Kuranishi complex}
	
	Recall that the deformation complex associated to an ASD connection $A$ is given by
	$$0\to\Omega^0(X,\mathfrak{g}_P)\xrightarrow{d_A}\Omega^1(X, \mathfrak{g}_P)\xrightarrow{d^+_A}\Omega^{2,+}(X, \mathfrak{g}_P)\to 0.$$
	The first differential $d_A$ is the infinitesimal action of the gauge transformation and the second differential $d_A^+$ is the linearization of the equation $F_A^+=0$. The cohomology groups $H_A^\bullet$ have their geometric meaning: $H_A^0=\mathrm{Ker}~d_A$ is zero if and only if $A$ is irreducible, $H^1_A=\mathrm{Ker}~d^+_A/\mathrm{Im}~d_A$ is the formal tangent space $T_{[A]}\mathcal{M}_{ASD}$, here $[A]$ denotes the equivalent class of $A$ under the gauge transformation, $H_A^2=\mathrm{Coker}~d_A^+$ is zero if and only if the map $A\mapsto F_A^+$ vanishes transversely at $A$, or equivalently, $d_A^+:\Omega^1(X, \mathfrak{g}_P)\to\Omega^{2,+}(X, \mathfrak{g}_P)$ is surjective. And the sum of $d_A^*:\Omega^1(X,\mathfrak{g}_P)\to\Omega^0(X,\mathfrak{g}_P)$ and $d_A^+:\Omega^1(X,\mathfrak{g}_P)\to\Omega^{2,+}(X, \mathfrak{g}_P)$:
	$$d_A^*+d_A^+:\Omega^1(X, \mathfrak{g}_P)\to\Omega^0(X, \mathfrak{g}_P)\oplus\Omega^{2,+}(X,\mathfrak{g}_P)$$
	is elliptic. The ellipticity of this operator is crucial to the construction of Donaldson's polynomial  invariants.
	
	For each chosen perturbation parameters $(t,\tau)$, the associated deformation complex of the perturbed VRVW map is
	$$\begin{aligned}
		0\to L^2_{k+1}(X,\Lambda^0\otimes\mathfrak{su}(2)_P)&\xrightarrow{d^0_{(A,B)}}L^2_k (X,\mathfrak{su}(2)_P\otimes\Lambda^1)\oplus L^2_k(X,\mathfrak{su}(2)_P\otimes\Lambda^{2,+})\\
		&\xrightarrow{d^1_{(A,B)}} L^2_{k-1}(X,\mathfrak{su}(2)_P\otimes\Lambda^{2,+})\oplus L^2_{k-1}(X,\mathfrak{su}(2)_P\otimes\Lambda^{2,+})\to 0,
	\end{aligned}$$
	where
	\begin{equation}
		d^0_{(A,B)}(\xi)=(d_A\xi, [B, \xi]) 
		\label{cf} \end{equation}
	is the linearization of the action of gauge group. The calculation of $d^1_{(A,B)}$ is more involved. First on 4-manifold $X$, let * denotes the Hodge-star operator, we have $d^*=-*d*$ on forms, and for $(A,B)\in\mathcal{C}_k(P)$ we have
	\begin{align*}
		&d_A^+d_A^*B\\
		=&-\frac{1}{2}(1+*)d_A*d_AB\\
		=&-\frac{1}{2}(1+*)d_A*(dB+[A,B])\\
		=&-\frac{1}{2}(1+*)d_A\left(*dB+*[A,B]\right)\\
		=&-\frac{1}{2}(1+*)\left(d\left(*dB+*[A,B]\right)+[A,*dB+*[A,B]]\right),
	\end{align*}
	hence the linearization of the map $(A,B)\to d_A^+d_A^*B$ is 
	\begin{align*}
		(a,b)\to&-\frac{1}{2}(1+*)(d(*db+*[a,B]+*[A,b])+[a,*dB+*[A,B]]\\
		&+[A,*db+*[a,B]+*[A,b]])\\
		=&-\frac{1}{2}(1+*)(d_A(*d_Ab+*[a,B])+[a,*d_AB])\\
		=&d_A^+d_A^*b-\frac{1}{2}(1+*)d_A*[a,B]-\frac{1}{2}(1+*)[a,*d_AB]\\
		=&d_A^+d_A^*b+\frac{1}{2}(1+*)(d_A^*[a,B]+[a,d_A^*B])\\
		=&d_A^+d_A^*b+\frac{1}{2}(1+*)[d_A^*a,B]\\
		=&d_A^+d_A^*b+[d_A^*a,B].
	\end{align*}
	Here we apply the equality $$d_A^*[a,B]=[d_A^*a,B]-[a,d_A^*B]$$ for $(A,B)\in\mathcal{C}_k (P)$ and $a\in L_k^2(X,\mathfrak{su}(2)_P\otimes\Lambda^{1})$. So
	\begin{equation} \begin{aligned}
			d^1_{(A,B)}(a,b)=\begin{pmatrix}
				d_A^+d_A^*b+[d_A^*a,B]+4t\langle B,B\rangle\langle B,b\rangle B+t\langle B,B\rangle^2b+4\tau[[b\centerdot B]\centerdot[B\centerdot B]]\\
				d_A^+a+\frac{1}{4}[b\centerdot B]
			\end{pmatrix}
		\end{aligned} \label{cg}
	\end{equation}
	is the linearization of the perturbed VRVW maps. 
	Also we have (cf. \cite{Fe, Ma})
	$$\begin{aligned}
		d^1_{(A,B)}\circ d^0_{(A,B)}(\xi)=\begin{pmatrix}
			[\xi,d_A^+d_A^*B+t\langle B,B\rangle^2B+\tau[[B\centerdot B]\centerdot[B\centerdot B]]]\\
			[\xi,F_A^++\frac{1}{8}[B\centerdot B]]\end{pmatrix}.
	\end{aligned}$$
	Therefore $d^1_{(A,B)}\circ d^0_{(A,B)}=0$ if and only if $\mathcal{PVRVW}(t,\tau,A, B)=0$, i.e. the sequence is a complex if and only if $[A, B]\in\mathcal{M}_{PVRVW}(t,\tau)$.
	
	The $L^2$ adjoint of $d^0_{(A,B)}$ is (cf. \cite{Ma})
	$$d^{0,*}_{(A,B)}(a,b)=d_A^*a+[b\cdot B],$$
	and the combined opertor
	$$\begin{aligned}
		\mathcal{D}_{(A,B)}&=d^{1}_{(A,B)}+ d^{0,*}_{(A,B)}:\begin{matrix} L^2_k(X,\mathfrak{su}(2)_P\otimes\Lambda^1) \\ \oplus \\ L^2_k(X,\mathfrak{su}(2)_P\otimes\Lambda^{2,+}) \end{matrix}\to\begin{matrix} L^2_{k-1}(X,\mathfrak{su}(2)_P\otimes\Lambda^{2,+}) \\ \oplus \\ L^2_{k-1}(X,\mathfrak{su}(2)_P\otimes\Lambda^{2,+})\\ \oplus \\ L^2_{k-1}(X,\mathfrak{su}(2)_P)\end{matrix}
	\end{aligned}$$
	differs from the following operator 
	\begin{equation} \mathcal{D}'_{(A,B)}(a,b)=  (d_A^+d_A^* b+[d_A^*a,B] , d_A^+ a , d_A^* a) 
		\label{cj} \end{equation}
	by zeroth-order terms. $\mathcal{D}'_{(A,B)}$ is the direct sum of two elliptic operators
	\begin{align*}
		d_A^+d_A^*+[d_A^*,B]:\begin{matrix} L^2_k(X,\mathfrak{su}(2)_P\otimes\Lambda^1) \\ \oplus \\ L^2_k(X,\mathfrak{su}(2)_P\otimes\Lambda^{2,+}) \end{matrix}&\to L^2_{k-1}(X,\mathfrak{su}(2)_P\otimes\Lambda^{2,+})\\
		(a,b)&\mapsto d_A^+d_A^* b+[d_A^*a,B]
	\end{align*}
	and
	\begin{align*}
		d_A^++d_A^*:\begin{matrix} L^2_k(X,\mathfrak{su}(2)_P\otimes\Lambda^1)\end{matrix}&\to \begin{matrix} L^2_{k-1}(X,\mathfrak{su}(2)_P\otimes\Lambda^{2,+}) \\ \oplus \\ L^2_{k-1}(X,\mathfrak{su}(2)_P) \end{matrix}\\
		a&\mapsto (d_A^+a,d_A^*a),
	\end{align*}
	hence a Fredholm operator. By the Sobolev multiplication theorem and the Rellich embedding theorem, there is a contiuous Sobolev multiplication map $L_k^2\times L_k^2\to L_k^2$ and the inclusion $L_k^2\subset L_{k-1}^2$ is compact when $k\geq 3$, so we have $\ind(\mathcal{D}_{(A,B)})=\ind(\mathcal{D}'_{(A,B)})$. The index of $d_A^+d_A^*+[d_A^*,B]$ is equal to that of $d_A^+d_A^*$ \cite[Corollary 7.9]{Lm}, which is 0, hence the index of $\mathcal{D}_{(A,B)}$ is equal to the index of $d_A^++d_A^*$, which is (see \cite[Theorem 6.1]{AHS2}\cite{D} and \cite[(2.1.40), (4.2.21) and (4.2.22)]{D2})
	$$\ind(\mathcal{D}_{(A,B)})=\ind(d_A^++d_A^*)=8\kappa(P)-3(1-b_1(X)+b_2^+(X)),$$
	here by the Chern-Weil theory, the characteristic number $\kappa(P)$ is defined by
	$$\kappa(P):=\frac{1}{8\pi^2}\int_{X}\tr(F_A^2)=
	\begin{cases}
		c_2(P),& G=SU(2),\\
		-\frac{1}{4}p_1(P),& G=SO(3).
	\end{cases}$$
	
	It follows that the above complex is an elliptic deformation complex for the perturbed VRVW equation with cohomology groups
	$$H_{(A,B)}^0:=\mathrm{Ker}d^{0}_{(A,B)},H_{(A,B)}^1:=\mathrm{Ker}d^{1}_{(A,B)}/\mathrm{Im}~d^{0}_{(A,B)},H_{(A,B)}^2:=\mathrm{Coker}d^{1}_{(A,B)}.$$
	Similarly, $H_{(A,B)}^0$ is the Lie algebra of the stabilizer of the triple $(A,B)$,  and $H_{(A,B)}^0=0$ if the stabilizer of $(A,B)$ is $Z(G)$, which means $A$ is irreducible, $H_{(A,B)}^1$ is the tangent space $T_{[A, B]}\mathcal{M}_{PVRVW}(t,\tau)$ and if  $\mathrm{Coker}(D\mathcal{PVRVW})_{(A,B)}=0$, then $H_{(A,B)}^2=0$ and $[A, B]$ is a regular point of $\mathcal{M}_{PVRVW}(t,\tau)$. We will prove later that every point of the   part of the moduli space $\mathcal{M}_{PVRVW}^*(t,\tau):=\mathcal{M}_{PVRVW}(t,\tau)\cap\mathcal{B}_k^\diamond(P)$ is regular and $\mathcal{M}_{PVRVW}^*(t,\tau)$ is a smooth manifold of dimension $\ind(\mathcal{D}'_{(A,B)})$.

	\section{Quadratic expansion of the perturbed VRVW map}
	
	We now find the quadratic expansion of $\mathcal{PVRVW}(t,\tau,A, B)$ with fixed perturbation parameters to build the regularity results (see  \cite[\S3]{FL} for the perturbed $PU(2)$-monopole equations and  \cite[\S3.2.2]{Ma} for Vafa-Witten equations). 
	Let $(A_0,B_0)\in\mathcal{C}_k(P)$ be smooth and $(a,b)\in L_k^2(X,\mathfrak{su}(2)_P\otimes\Lambda^1)\oplus L_k^2(X,\mathfrak{su}(2)_P\otimes\Lambda^{2,+})$, then
	\begin{align*}
		&\mathcal{PVRVW}(t,\tau,A_0+a,B_0+b)\\
		=&\begin{pmatrix}d_{A_0+a}^+d_{A_0+a}^*(B_0+b)+t\langle (B_0+b),(B_0+b)\rangle^2(B_0+b)+\tau[[(B_0+b)\centerdot (B_0+b)]\centerdot[(B_0+b)\centerdot (B_0+b)]]\\
			F_{A_0+a}^++\frac{1}{8}[(B_0+b)\centerdot (B_0+b)]\end{pmatrix}\\
		=&\mathcal{PVRVW}(t,\tau,A_0,B_0)+d_{(A_0,B_0)}^1(a,b)\\
		&+\begin{pmatrix}-d_{A_0}^+[a,b]-[a,*d_{A_0}b]^+-[a,*[a,B_0+b]]^++4t\langle B_0,b\rangle^2B_0+2\tau[[b\centerdot b]\centerdot[B_0\centerdot B_0]]+\ldots\\
			\frac{1}{2}[a\wedge a]^++\frac{1}{8}[b\centerdot b]\end{pmatrix} \\
		=&\mathcal{PVRVW}(t,\tau,A_0,B_0)+d_{(A_0,B_0)}^1(a,b)+\{(a,b),(a,b)\},
	\end{align*}
	where
	$$\begin{aligned}
		&\{(a,b),(a,b)\}\\
		:=&\begin{pmatrix}-d_{A_0}^+[a,b]-[a,*d_{A_0}b]^+-[a,*[a,B_0+b]]^++4t\langle B_0,b\rangle^2B_0+2\tau[[b\centerdot b]\centerdot[B_0\centerdot B_0]]+\ldots\\
			\frac{1}{2}[a\wedge a]^++\frac{1}{8}[b\centerdot b]\end{pmatrix},
	\end{aligned}$$
	a quadratic form of $(a,b)$. Given $(A_0,B_0)\in\mathcal{C}_k(P)$ and $(u_0,v_0)\in\mathcal{C}_{k-1}'(P)$, we find solutions to the inhomogeneous equation $\mathcal{PVRVW}(t,\tau,A_0+a,B_0+b)=(u_0,v_0)$ for the pair $(a,b)$. To make the equation elliptic, we impose the gauge-fixing condition
	$$d^{0,*}_{(A,B)}(a,b)=w,$$
	then the equation we considered is an elliptic equation of the following form
	\begin{equation} \mathcal{D}_{(A_0,B_0)}(a,b)+\{(a,b),(a,b)\}=(w,u,v),
		\label{ccc}
	\end{equation}
	here $(u,v)=(u_0,v_0)-\mathcal{PVRVW}(t,\tau,A_0,B_0)$. Equation \eqref{ccc} is called \emph{the perturbed   VRVW equation under Coulomb gauge}.
	
	Mimicking the regular results of  \cite[\S3]{FL} and  \cite[\S3.3]{Ma}, we have
	\begin{theorem} \label{tj} {\bf (Global estimate for $L_1^2$ solutions to the inhomogeneous perturbed VRVW plus Coulomb slice equations, cf.  \cite[Corollary 3.4]{FL}, \cite[Theorem 3.3.1]{Ma})\/}
		Let $(X,g)$ be a compact, oriented and smooth Riemannian 4-manifold, $P\to X$ a principle $SU(2)$- or $SO(3)$-bundle on $X$ and let $(A_0,B_0)$ be a $C^{\infty}$ configuration in $\mathcal{C}(P)$. Then there is a positive constant $\epsilon=\epsilon(A_0,B_0)$ such that if $(a,b)$ is an $L_1^2$ solution to the  equation \eqref{ccc}, where $(w,u,v)$ is in $L_k^2$ and $\vert \vert (a,b)\vert \vert _{L^4(X)}<\epsilon$, and $k\geq 3$ is an integer, then $(a,b)\in L_{k+1}^2$ and there is a polynomial $Q_k(x,y)$, with positive real coefficients, depending at most on $(A_0,B_0),k$ such that $Q_k(0,0)=0$ and
		$$\vert \vert (a,b)\vert \vert _{L_{k+1,A_0}^2(X)}\leq Q_k\Big(\vert \vert (w,u,v)\vert \vert _{L_{k,A_0}^2(X)},\vert \vert (a,b)\vert \vert _{L^2}\Big).$$
		In particular, if $(w,u,v)$ is in $C^r$ then $(a,b)$ is in $C^{r+1}$. 
	\end{theorem}
	By applying Theorem \ref{tj}, we can obtain the following global regularity theorem.
	\begin{theorem} \label{tk}
		{\bf  (Global regularity of $L_k^2$ solutions to the perturbed VRVW equations for $k\geq 3$, cf.  \cite[Proposition 3.7]{FL}, \cite[Theorem 3.3.2]{Ma})} Let $(X,g)$ be a compact, oriented and smooth Riemannian 4-manifold and $P\to X$ a principle $SU(2)$- or $SO(3)$-bundle on $X$. Let $k\geq 3$ be an integer and suppose that $(A,B)$ is an $L_k^2$ solution to $\mathcal{PVRVW}(t,\tau,A, B)=0$ for fixed $C^r$ perturbation parameters $(t,\tau)$, then there is a gauge transformation $\zeta\in L_{k+1}^2(\mathcal{G}_P)$ such that $\zeta\cdot (A,B)$ is $C^\infty$ over $X$.
	\end{theorem}
	
	\section{Transversality of the perturbed VRVW equations}\label{trans}
	
	In this section we verify the perturbed   VRVW map \eqref{cd}, viewed as a section of the Banach vector bundle $\overline{\mathcal{E}}_k$ over $\mathcal{T}^r\times\mathcal{B}_k^\diamond(P)$, is transverse to the zero section of $\overline{\mathcal{E}}_k$ over $\mathcal{T}^r\times\mathcal{B}_k^\diamond(P)$. 
	
	\subsection{The linearization of the perturbed VRVW map}
	
	Fix a perturbation parameter $(t,\tau)\in\mathcal{T}^r$, if $\Gamma:=(t,\tau,A, B) \in \mathcal{T}^r\times\mathcal{B}_k(P)$ satisfies $\mathcal{PVRVW}(\Gamma)=0$, then the linearization of the map $\mathcal{PVRVW}$ at $\Gamma$ is 
	\begin{align*}
		&(D\mathcal{PVRVW})_{\Gamma}(\delta t,\delta\tau,a,b)\\
		=&\begin{pmatrix}(D\mathcal{PVRVW}_1)_{\Gamma}(\delta t,\delta\tau,a,b)\\
			(D\mathcal{PVRVW}_2)_{\Gamma}(\delta t,\delta\tau,a,b)\end{pmatrix}\\
		=&\begin{pmatrix}
			d_A^+d_A^*b+[d_A^*a,B]+4t\langle B,B\rangle\langle B,b\rangle B+\delta t\langle B,B\rangle^2B+4\tau[[b\centerdot B]\centerdot[B\centerdot B]]+\delta\tau[[B\centerdot B]\centerdot[B\centerdot B]]\\
			d_A^+a+\frac{1}{4}[b\centerdot B]
		\end{pmatrix},
	\end{align*}
	where $(\delta t,\delta\tau,a,b)\in\mathcal{T}^r\times T_{[A, B]}\mathcal{M}_{PVRVW}^*(t,\tau)$ (It's easy to see that $T_{(t,\tau)}\mathcal{T}^r=\R\times C^r(GL(\Lambda^{2,+}))$).   
	Note that the full differential $(D\mathcal{PVRVW})_{\Gamma}$ differs from the parameter fixed differential $d^1_{(A, B)}$ in \eqref{cg} by bounded linear terms in $(\delta t,\delta\tau)$. This fact, together with the estimate in Theorem~\ref{tj}, implies that $(D\mathcal{PVRVW})_{\Gamma}$ has closed range. Hence
	$$\mathrm{Ran}(D\mathcal{PVRVW})_{\Gamma} \not= \mathcal{C}'_{k-1}(P)$$
	if and only if there is a nonzero pair $(\phi,\psi)\in\mathcal{C}'_{k-1}(P)=L_{k-1}^2(X,\mathfrak{su}(2)_P\otimes\Lambda^1)\oplus L_{k-1}^2(X,\mathfrak{su}(2)_P\otimes\Lambda^{2,+})$ such that $\forall(\delta t,\delta\tau,a,b)\in\mathcal{T}^r\times T_{[A, B]}\mathcal{M}_{PVRVW}^*(t,\tau)$ we have
	\begin{equation}\label{yy}\begin{aligned}
			&\langle(D\mathcal{PVRVW})_{\Gamma}(\delta t,\delta\tau,a,b),(\phi,\psi)\rangle_{L^2}\\
			=&\langle(D\mathcal{PVRVW}_1)_{\Gamma}(\delta t,\delta\tau,a,b),\phi\rangle_{L^2}\\
			+&\langle(D\mathcal{PVRVW}_2)_{\Gamma}(\delta t,\delta\tau,a,b),\psi\rangle_{L^2}=0.
	\end{aligned}\end{equation}
	The above formula implies $(\phi,\psi)\in\mathrm{Ker}(D\mathcal{PVRVW})_{\Gamma}^*$ \big(where $(D\mathcal{PVRVW})_{\Gamma}^*$ is the $L^2$ adjoint operator of $(D\mathcal{PVRVW})_{\Gamma}$\big), then applying elliptic regularity for the Laplacian $(D\mathcal{PVRVW})_{\Gamma}(D\mathcal{PVRVW})_{\Gamma}^*$ with $C^{r-1}$ coefficients implies that $(\phi,\psi)$ is $C^{r+1}$(cf.  \cite[\S5]{FL}). And Aronszajn's theorem (cf.  \cite[Remark 3]{Ar}, \cite[Theorem 1.8]{Kaz}) implies that $(\phi,\psi)$ in $\mathrm{Ker}\left((D\mathcal{PVRVW})_{\Gamma}(D\mathcal{PVRVW})_{\Gamma}^*\right)$ has the unique continuation property (cf.  \cite[Lemma 5.9]{FL}). Therefore, to prove that $(\phi,\psi)\equiv 0$ on $X$, it is only necessary to prove that $(\phi,\psi)$ is zero on an open subset of $X$.

	\subsection{Establishment of the transversality}
	
	Before establishing the transversality, we prove the following lemma first:
	
	\begin{lemma}\label{xzz}Let $(X,g)$ be an oriented, compact and smooth Riemannian 4-manifold and $P\to X$ a principle $SU(2)$- or $SO(3)$-bundle on $X$.  For every perturbation parameter $(t,\tau)\in\mathcal{T}^r$, if $[A, B]\in\mathcal{B}_k^\diamond(P)$ is a solution to the perturbed VRVW equation
		$$\left\{
			\begin{aligned}
				&d_A^+d_A^*B+t\langle B,B\rangle^2B+\tau[[B\centerdot B]\centerdot[B\centerdot B]]=0\\
				&F_A^++\frac{1}{8}[B\centerdot B]=0,\\
			\end{aligned}\right.$$
		then $\rank(B)=3$ on $X$.
	\end{lemma}
	
	\begin{proof}
		Choose a solution $[A, B]\in\mathcal{B}_k^\diamond(P)$ to \eqref{yz} such that $\rank(B)\leq 2$ on $X$, then \cite[\S4.1.1]{Ma} $$[[B\centerdot B]\centerdot[B\centerdot B]]=0$$ on $X$.  Take inner product with $B$ in the first equation of \eqref{yz}, we have
		$$\langle d_A^+d_A^*B+t\langle B,B\rangle^2B,B\rangle_{L^2}=0,$$
		so
		$$\vert \vert d_A^*B\vert \vert ^2_{L^2}+t\vert \vert B\vert \vert _{L^2}^6=0.$$
		Then $t>0$  implies 
		$$\vert \vert d_A^*B\vert \vert ^2_{L^2}=0,~\vert \vert B\vert \vert _{L^2}^6=0$$
		on $X$ and hence $B\equiv 0$ on $X$, contradicts to the assumption $[A, B]\in\mathcal{B}_k^\diamond(P)$, so $B$ is rank 3 on $X$. 
	\end{proof}
	
	Now we prove that $(\phi,\psi)$ is zero on an open subset of $X$. For a perturbation parameter $(t,\tau)\in\mathcal{T}^r$, let $[A, B]\in\mathcal{B}_k^\diamond(P)$ be a solution to the corresponding perturbed VRVW equation \eqref{yz}. Lemma \ref{xzz} implies that there is an open subset $V\subset X$ such that $B$ is rank 3 on $V$. Then we can see that $[[B\centerdot B]\centerdot[B\centerdot B]]$ is also rank 3 on $V$\cite[\S4.1.1]{Ma}. In equations \eqref{yy}, set $(\delta t,a,b)=0$ we have
	$$\langle\delta\tau[[B\centerdot B]\centerdot[B\centerdot B]],\phi\rangle_{L^2}=0,~\forall\delta\tau\in C^r(\mathfrak{gl}(\Lambda^{2,+}))$$
	and \cite[Lemma 2.3]{Fe} implies $\phi\equiv 0$ on $V$. Then set $(\delta t,\delta\tau,a)=0$ in equations \eqref{yy} we have
	$$\langle b,[B\centerdot\psi]\rangle_{L^2}=\langle[b\centerdot B],\psi\rangle_{L^2}=0,~\forall b\in L^2_k(X,\mathfrak{su}(2)_P\otimes\Lambda^{2,+}),$$
	so on $V$, $[B\centerdot\psi]=0$ and Lemma \ref{xx} implies $\psi\equiv 0$ on $V$. Hence we have $(\phi,\psi)\equiv 0$ on $V$, so $(\phi,\psi)\equiv 0$ on $X$ by the unique continuation for the Laplacian $(D\mathcal{PVRVW})_{\Gamma}(D\mathcal{PVRVW})_{\Gamma}^*$ and we finish the establishment of the transversality.
	
	By the slice result (cf.  \cite[Proposition 2.8]{FL}), $T_{[A, B]}\mathcal{B}_k^\diamond(P)$ may be identified with $\mathrm{Ker}~d_{(A,B)}^{0,*}$. For $(a,b)\in\mathrm{Ker}~d_{(A,B)}^{0,*}$, we have
	$$\begin{aligned}
		(D\mathcal{PVRVW})_{(A,B)}(0,0,a,b)&=d_{(A,B)}^1(a,b)\\
		&=(d_{(A,B)}^{0,*}+d_{(A,B)}^1)(a,b),\end{aligned}$$
	and it's easy to see that the differential $(D\mathcal{PVRVW})_{(A,B)}\vert _{\{(0,0)\}\times T\mathcal{C}_k^\diamond(P))}$ is Fredholm, where ${\{0\}\times T\mathcal{C}_k^\diamond(P)}=T(\{(t,\tau)\}\times\mathcal{B}_k^\diamond(P))$. Thus $\mathcal{PVRVW}$ is a Fredholm section when restricted to the fixed-parameter fibers $\{(t,\tau)\}\times\mathcal{C}_k^\diamond(P)\subset\mathcal{T}^r\times\mathcal{C}_k^\diamond(P)$ where $(t,\tau)\in\mathcal{T}^r$, so the Sard-Smale theorem (cf. \cite[Proposition 4.12]{Fe}, \cite[Proposition 4.3.11]{D2}) implies that there is a first-category subset $\mathcal{T}^r_{fc}\subset\mathcal{T}^r$ such that the zero sets in $\mathcal{C}_k^\diamond(P)$ of $\mathcal{PVRVW}(t,\tau,\cdot)$ are regular (note the transversality) for all perturbations $(t,\tau)\in\mathcal{T}^r-\mathcal{T}^r_{fc}$.

	In summary, we have the following theorem:
	
	\begin{theorem}\label{ts}Let $(X,g)$ be a compact, oriented and smooth Riemannian 4-manifold, $P\to X$ a principle $G$-bundle. If $\kappa(P)\geq 3/8(1-b_1(X)+b_2^+(X))$, then there is a first-category subset $\mathcal{T}^r_{fc}\subset\mathcal{T}^r$ such that for all $(t,\tau)$ in $\mathcal{T}^r-\mathcal{T}^r_{fc}$ the following holds: The zero set of the section $\mathcal{PVRVW}(t,\tau,\cdot)$ in $\mathcal{C}_k^\diamond(P)$ is regular and the moduli space $\mathcal{M}_{PVRVW}^*(t,\tau):=\mathcal{M}_{PVRVW}(t,\tau)\cap\mathcal{B}_k^\diamond(P)=\big(\mathcal{PVRVW}(t,\tau,\cdot)^{-1}(0)/\mathcal{G}_{k+1}(P)\big)\cap\mathcal{B}_k^\diamond(P)$ is a smooth manifold of dimension $\ind(\mathcal{D}_{(A,B)})$, which is
		$$\ind(\mathcal{D}_{(A,B)})=8\kappa(P)-3(1-b_1(X)+b_2^+(X)).$$
	\end{theorem}
	
	\section{A priori bound on the solutions}
	
	We show that there is \emph{a priori} boundness for the solutions $[A,B]$ of \eqref{yz}, which is an indispensable step in constructing the Ulhenbeck closure of the perturbed VRVW moduli spaces.
	
	\begin{theorem}\label{ulh}Let $(X,g)$ be a compact, oriented and smooth Riemannian 4-manifold, $P\to X$ a principle $G$-bundle. If $\kappa(P)\geq 3/8(1-b_1(X)+b_2^+(X))$, then for any generic parameter $(t,\tau)\in\mathcal{T}^r-\mathcal{T}^r_{fc}$ there are constants $C_{t,\tau,X}$ and $K_{t,\tau,X}$ depend only on $t,\tau$ and $X$ such that for any solution $[A,B]$ of \eqref{yz}, we have $\vert\vert B\vert\vert_{L^\infty} \leq C_{t,\tau,X}$  and $\vert\vert d_AB\vert\vert_{L^2}\leq K_{t,\tau,X}$. Moreover, for fixed $\tau$, $\lim_{t\to\infty}K_{t,\tau,X}=\lim_{t\to\infty}C_{t,\tau,X}=0$.
	\end{theorem}
	\begin{proof}
		Fix a generic parameter $(t,\tau)\in\mathcal{T}^r-\mathcal{T}^r_{fc}$, for $\forall x\in X$, in a neighborhood $U$ of $x$ according to the singular value decomposition for $\mathfrak{su}(2)\otimes\Lambda^{2,+}\mathbb{R}^4$ in  \cite[\S 4.1.1]{Ma}, there exist oriented orthonormal basis $\{e^1,e^2,e^3,e^4\}$ for $\mathbb{R}^4$ and a basis $\{\eta_1,\eta_2,\eta_3\}$ for $\mathfrak{su}(2)$   such that $[\eta_1,\eta_2]=2\eta_3$ and cyclic permutations, $B=B_{1}\eta_1\otimes(e^1\wedge e^2+e^3\wedge e^4)+B_{2}\eta_2\otimes(e^1\wedge e^3+e^4\wedge e^2)+B_{3}\eta_3\otimes(e^1\wedge e^4+e^2\wedge e^3)$, 
		where $B_{1},B_{2},B_{3}\in L_k^2(U)$. For simplicity denote $\sigma^1:=e^1\wedge e^2+e^3\wedge e^4$, $\sigma^1:=e^1\wedge e^3+e^4\wedge e^2$ and $\sigma^3:=e^1\wedge e^4+e^2\wedge e^3$, then it's not hard to see that $\sigma^1\centerdot\sigma^2=-2\sigma^3$ and cyclic permutations. Then we have
		\begin{align*}
			B=&B_1\eta_1\sigma^1+B_2\eta_2\sigma^2+B_3\eta_3\sigma^3,\\
			[B\centerdot B]=&-8B_2B_3\eta_1\sigma^1-8B_3B_1\eta_2\sigma^2-8B_1B_2\eta_3\sigma^3,\\
			[[B\centerdot B]\centerdot[B\centerdot B]]=&-512B_1^2B_2B_3\eta_1\sigma^1-512B_1B_2^2B_3\eta_2\sigma^2-512B_1B_2B_3^2\eta_3\sigma^3,\\
			\vert B\vert ^2=&\langle B,B\rangle=2(B_1^2+B_2^2+B_3^2).
		\end{align*}
		Let $\tau=\begin{pmatrix}
			\tau_{11}&\tau_{12}&\tau_{13}\\
			\tau_{21}&\tau_{22}&\tau_{23}\\
			\tau_{31}&\tau_{32}&\tau_{33}\\
		\end{pmatrix}$ be the representation under the basis $\{\sigma^1,\sigma^2,\sigma^3\}$, then	
		\begin{align*}
			\tau[[B\centerdot B]\centerdot[B\centerdot B]]
			=&-512B_1^2B_2B_3\eta_1(\tau_{11}\sigma^1+\tau_{12}\sigma^2+\tau_{13}\sigma^3)\\
			&-512B_1B_2^2B_3\eta_2(\tau_{21}\sigma^1+\tau_{22}\sigma^2+\tau_{23}\sigma^3)\\
			&-512B_1B_2B_3^2\eta_3(\tau_{31}\sigma^1+\tau_{32}\sigma^2+\tau_{33}\sigma^3),\\
			-\langle B\cdot\tau[[B\centerdot B]\centerdot[B\centerdot B]]\rangle=&1024B_1B_2B_3(\tau_{11}B_1^2+\tau_{22}B_2^2+\tau_{33}B_3^2)\\
			\leq&1024(\frac{B_1^2+B_2^2+B_3^2}{3})^{\frac{3}{2}}\frac{\max\{\vert \tau_{11}\vert ,\vert \tau_{22}\vert ,\vert \tau_{33}\vert \}}{2}2(B_1^2+B_2^2+B_3^2)\\
			=&\lambda_\tau\vert B\vert ^5
		\end{align*}
		for some constant $\lambda_\tau$ depends only on $\tau$.

		Denote by $\Delta_g$ the Laplace-Beltrami opertor on $(X,g)$, for any solution $[A, B]\in\mathcal{B}_k^\diamond(P)$ to the equation \eqref{yz}, applying \cite[(B.19)]{Ma} we have
		\begin{align*}
			&\Delta_g\vert B\vert ^2\\
			\leq&\Delta_g\vert B\vert ^2+2\vert \nabla_AB\vert ^2\\
			=&2\langle B\cdot\nabla_A^*\nabla_AB\rangle\\
			=&4\langle B\cdot d_A^+d_A^*B\rangle-\frac{2}{3}(s\vert B\vert ^2-6W^+\cdot\langle B\odot B\rangle)+2\langle B\cdot[F_A^+\centerdot B]\rangle\\
			\leq&-4t\vert B\vert ^6-4\langle B\cdot\tau[[B\centerdot B]\centerdot[B\centerdot B]]\rangle+\lambda_X\vert B\vert ^2+2\langle F_A^+\cdot[B\centerdot B]\rangle\\
			\leq&-4t\vert B\vert ^6+4\lambda_\tau\vert B\vert ^5+\lambda_X\vert B\vert ^2-\frac{1}{4}\vert [B\centerdot B]\vert ^2\\
			\leq&-4t\vert B\vert ^6+4\lambda_\tau\vert B\vert ^5+\lambda_X\vert B\vert ^2
		\end{align*}
		for another constant $\lambda_X$ relies only on the topological imformation of $X$. Since $X$ is compact, $\vert B\vert ^2$ attains its maximum at a point $x_0\in X$. At $x_0$,
		$$\Delta_g\vert B\vert ^2=-\sum_{i}\partial_i\partial_i\vert B\vert ^2\geq 0,$$
		hence we have
		$$-4t\vert B(x_0)\vert ^6+4\lambda_\tau\vert B(x_0)\vert ^5+\lambda_X\vert B(x_0)\vert ^2\geq 0$$
		which is equivalent to
		$$-4t\vert B(x_0)\vert ^4+4\lambda_\tau\vert B(x_0)\vert ^3+\lambda_X\geq 0.$$
		the above inequality and $t>0$ imply $\vert B(x_0)\vert \leq C_{t,\tau,X}$ for some constant $C_{t,\tau,X}$ dependents only on $t,\tau$ and $X$ and so $\vert\vert B\vert\vert_{L^\infty} \leq C_{t,\tau,X}$.
		
		When $t>2\sqrt[3]{4\lambda_\tau^4/\lambda_X}$ and if $\vert B(x)\vert>\sqrt[4]{\lambda_X/(2t)}>\lambda_\tau/t$ for some $x\in X$, then
		\begin{align*}
			&4t\vert B(x)\vert ^4-4\lambda_\tau\vert B(x)\vert ^3-\lambda_X\\
			=&4\vert B(x)\vert^3(t\vert B\vert-\lambda_\tau)-\lambda_X\\
			>&4(\sqrt[4]{\lambda_X/(2t)})^3(t\sqrt[4]{\lambda_X/(2t)}-\lambda_\tau)-\lambda_X\\
			=&\lambda_X-4\lambda_\tau(\sqrt[4]{\lambda_X/(2t)})^3>0,
		\end{align*}
		so we must have $C_{t,\tau,X}\leq\sqrt[4]{\lambda_X/(2t)}$, hence $\lim_{t\to\infty}C_{t,\tau,X}=0$.
		
		Take inner product with $B$ in the first equation of \eqref{yz}, we have
		$$\vert\vert d_A^*B\vert\vert_{L^2}^2+t\vert\vert B\vert\vert_{L^2}^6+\langle\tau[[B\centerdot B]\centerdot[B\centerdot B]]\cdot B\rangle_{L^2}=0,$$
		hence
		\begin{align*}
			\vert\vert d_A^*B\vert\vert_{L^2}^2&=-t\vert\vert B\vert\vert_{L^2}^6-\langle\tau[[B\centerdot B]\centerdot[B\centerdot B]]\cdot B\rangle_{L^2}\\
			&\leq-\langle\tau[[B\centerdot B]\centerdot[B\centerdot B]]\cdot B\rangle_{L^2}\\
			&\leq\int_X\lambda_\tau\vert B\vert ^5d\mathrm{vol}\\
			&\leq\lambda_\tau C_{t,\tau,X}^5\mathrm{vol}(X).
		\end{align*}
		Denote $$K_{t,\tau,X}:=\sqrt{\lambda_\tau C_{t,\tau,X}^5\mathrm{vol}(X)},$$we have $\vert\vert d_AB\vert\vert_{L^2}=\vert\vert d_A^*B\vert\vert_{L^2}\leq K_{t,\tau,X}$, and $\lim_{t\to\infty}C_{t,\tau,X}=0$ implies $\lim_{t\to\infty}K_{t,\tau,X}=0$.
	\end{proof}
	
	\section{Ulhenbeck closure of the moduli spaces}
	
	Given the a priori boundness result Theorem~\ref{ulh}, by applying the methods of the proof of \cite[Theorem 4.10]{FL}, we have the following removable singularities theorem for perturbed VRVW equations.
	
	\begin{theorem}\label{rm}
		{\bf (cf.  \cite[Theorem 4.10]{FL} \cite[Theorem 3.4.1]{Ma}).} Let $B(x,r)\subset X$ be a geodesic ball and  $P \to B(x,r)\backslash\{x\}$ be a principal $SU(2)$ or $SO(3)$-bundle. Suppose $(A,B)$ is a $C^\infty$ solution to the perturbed VRVW equations \eqref{yz} for $P$ over the punctured ball $B(x,r)\backslash\{x\}$ with
		$$\int_{B(x,r)\backslash\{x\}}(\vert F_A\vert^2+\vert d_AB\vert^2+\vert B\vert^4)d\mathrm{vol}\leq\mathrm{const.}$$
		Then there is a principal bundle $\tilde{P}\to B(x,r)$, a $C^\infty$ solution $(\tilde{A},\tilde{B})$ to the perturbed VRVW equations \eqref{yz} for $\tilde{P}$ over $B(x,r)$, and a $C^\infty$ bundle isomorphism $u:P\vert_{B(x,r)\backslash\{x\}}\to\tilde{P}\vert_{B(x,r)\backslash\{x\}}$ such that
		$$u^*(\tilde{A},\tilde{B})=(A,B)\text{ over }B(x,r)\backslash\{x\}.$$		
	\end{theorem}
	Later, by applying the Chern-Weil identity we will show that 
	$$\vert\vert F_A\vert\vert^2_{L^2}:=\int_X \vert F_A\vert^2d\mathrm{vol}<\infty,$$
	then combine with Theorem~\ref{ulh} we have
	$$\int_{X\backslash\{x_1,\ldots,x_m\}}(\vert F_A\vert^2+\vert d_AB\vert^2+\vert B\vert^4)d\mathrm{vol}<\infty$$
	for $\forall m\in\N^*$ and $\{x_1,\ldots,x_m\}\subset X$. The above bound is crucial for constructing the Ulhenbeck closure.
	
	Now, we are in place constructing the Ulhenbeck compactification of the  perturbed VRVW moduli space $\mathcal{M}_{PVRVW,\kappa(P)}(t,\tau)$. For $P_m\to X$ a principle $G$-bundle with characteristic number $\kappa(P)=m$, let $\mathcal{T}^r_{m,fc}$ denotes the first-category set defined in Theorem \ref{ts}, and then define
	$${\mathcal{T}^r}':=\bigcup_{k=-\infty}^\infty\mathcal{T}^r_{m,fc}.$$
	${\mathcal{T}^r}'$ is a countable union of first-category sets, hence it's also a set of first-category. Then for any $(t,\tau)\in\mathcal{T}^r-{\mathcal{T}^r}'$ and $P_m\to X$, we define the set of \emph{ideal solutions of perturbed VRVW equation} $\mathcal{IM}_{PVRVW,m}(t,\tau)$ to be
	$$\mathcal{IM}_{PVRVW,m}(t,\tau):=\bigcup_{l=0}^\infty\mathcal{M}_{PVRVW,m-l}(t,\tau)\times\mathrm{Sym}^l(X).$$
	A sequence $\{[A_i,B_i,\textbf{x}_i]\}\in\mathcal{IM}_{PVRVW,m}(t,\tau)$ \emph{converges to} $[A_0,B_0,\textbf{x}_0]\in\mathcal{M}_{PVRVW,m-l_0}(t,\tau)\times\mathrm{Sym}^{l_0}(X)$ if for some (or equivalently any) choice of smooth representatives $(A_i,B_i)\in\mathcal{M}_{PVRVW,m-l_i}(t,\tau)$ the following hold:
	
	$\bullet$ There is a sequence of smooth bundle isomorphisms $g_i:P_{m-l_i}\vert _{X\setminus\textbf{x}_0}\to P_{m-l_0}\vert _{X\setminus\textbf{x}_0}$ such that
	$g_i(A_i, B_i)$ converges in $C^\infty$ to $(A_0, B_0)$ over $X\setminus\textbf{x}_0$.
	
	$\bullet$ The sequence $\vert F_{A_i}\vert ^2+8\pi^2\sum_{x\in\textbf{x}_i}\delta_x$ converges in the weak-* topology on measures to $\vert F_{A_0}\vert ^2+8\pi^2\sum_{x\in\textbf{x}_0}\delta_x$.
	
	If $[A,B]\in\mathcal{M}_{PVRVW,m-l}(t,\tau)$, $l\in\mathbb{N}$, from the proof of Theorem \ref{ulh} we can see that locally
	\begin{align*}
		B=&B_1\eta_1\sigma^1+B_2\eta_2\sigma^2+B_3\eta_3\sigma^3,\\
		[B\centerdot B]=&-8B_2B_3\eta_1\sigma^1-8B_3B_1\eta_2\sigma^2-8B_1B_2\eta_3\sigma^3,
	\end{align*}
	hence	
	\begin{align*}
		\vert B\vert ^2=&\langle B,B\rangle=2(B_1^2+B_2^2+B_3^2),\\
		\vert [B\centerdot B]\vert ^2=&\langle [B\centerdot B],[B\centerdot B]\rangle=128(B_2^2B_3^2+B_3^2B_1^2+B_1^2B_2^2).
	\end{align*}
	The inequality $(B_1^2+B_2^2+B_3^2)^2\geq 3(B_2^2B_3^2+B_3^2B_1^2+B_1^2B_2^2)$ implies that pointwisely we have
	$$\left\vert \frac{1}{8}[B\centerdot B]\right\vert ^2\leq\frac{1}{6}\vert B\vert ^4,$$
	so
	$$\left\vert \left\vert \frac{1}{8}[B\centerdot B]\right\vert \right\vert ^2_{L^2}=\int_{X}\left\vert \frac{1}{8}[B\centerdot B]\right\vert ^2d\mathrm{vol}\leq\int_X\frac{1}{6}\vert B\vert ^4d\mathrm{vol}\leq\frac{\mathrm{vol}(X)}{6}C_{t,\tau,X}^4.$$
	Then the Chern-Weil identity implies
	\begin{align*}
		0\leq \vert \vert F_A^-\vert \vert ^2_{L^2}&=8\pi^2 (m-l)+\vert \vert F_A^+\vert \vert ^2_{L^2}\\
		&=8\pi^2 (m-l)+\left\vert \left\vert \frac{1}{8}[B\centerdot B]\right\vert \right\vert ^2_{L^2}\\
		&\leq  8\pi^2 (m-l)+\frac{\mathrm{vol}(X)}{6}C_{t,\tau,X}^4,
	\end{align*}
	so we must have
	$$l\leq m+\frac{\mathrm{vol}(X)C_{t,\tau,X}^4}{48\pi^2}$$
	which means $\mathcal{IM}_{PVRVW,m}(t,\tau)$ is in fact a finite union. Since $\vert\vert B\vert\vert_{L^\infty} \leq C_{t,\tau,X}$, $\vert \vert F_A^+\vert \vert _{L^\infty}=\vert \vert \frac{1}{8}[B\centerdot B]\vert \vert ^2_{L^\infty}\leq\frac{1}{6}\vert \vert B\vert \vert ^4_{L^\infty}$ is also bounded on $X$. Moreover
	\begin{align*}
		\vert \vert F_A\vert \vert ^2_{L^2}&=\vert \vert F_A^+\vert \vert ^2_{L^2}+\vert \vert F_A^-\vert \vert ^2_{L^2}\\
		&=8\pi^2(m-l)+2\vert \vert F_A^+\vert \vert ^2_{L^2}\\
		&\leq 8\pi^2(m-l)+\frac{\mathrm{vol}(X)}{3}C_{t,\tau,X}^4=\mathrm{const.}
	\end{align*} 
	Then by applying a similar argument as in \cite{FL}, we have the following theorem.
	\begin{theorem}  \label{final}
		{\bf (cf.  \cite[Theorem 4.20]{FL} \cite[Theorem 3.5.2]{Ma}).} Let $X$ be a closed, oriented, smooth Riemannian 4-manifold with Riemannian metric $g$, and $P\to X$ a principal $G=SU(2)$- or $SO(3)$-bundle with characteristic number 
		$$\kappa\geq-\frac{\mathrm{vol}(X)C_{t,\tau,X}^4}{48\pi^2},$$ then the Ulhenbeck closure $\overline{\mathcal{M}}_{PVRVW,\kappa}(t,\tau)\subset\mathcal{IM}_{PVRVW,\kappa}(t,\tau)$   is  sequentially compact.
	\end{theorem}	
	
	\section{The product bundle case}

	In this section, we prove Theorem \ref{m3}. We will need the following theorem, which is a special case of \cite[Theorem 1.3]{Ulh}. See \cite[Theorem 2.3.8]{D2}\cite[Corollary 3.15]{FL} for analogous results.
	\begin{theorem}\label{pb}
		 Let $X=B^4$, $G=SU(2)$  or $SO(3)$, $P=X\times G$, $\tilde{A}\in \mathcal{A}_k (P)$. Then there exists constants $K_4,C_4>0$ such that if $\vert\vert F_{\tilde{A}}\vert\vert_{L^2}\leq K_4$, then $\tilde{A}$ is gauge equivalent to a connection $A$ such that $d^*A=0$ and $\vert\vert A\vert\vert_{L_k^2}\leq C_4\vert\vert F_{A}\vert\vert_{L^2}$.
	\end{theorem}
	
	If $P=X\times G$, the product bundle, then $\kappa(P)=0$, and if $[A,B]\in\mathcal{M}^*_{PVRVW,0}(t,\tau)$, from the former section we have
	$$\vert \vert F_A\vert \vert_{L^2}\leq\sqrt{\frac{\mathrm{vol}(X)}{3}}C_{t,\tau,X}^2.$$
	Theorem \ref{ulh} implies that there is a $c_\tau>0$ such that for any generic $(t,\tau)\in\mathcal{T}^r-\mathcal{T}^r_{fc}$ with $t\geq c_\tau$, 
	$$\sqrt{\frac{\mathrm{vol}(X)}{3}}C_{t,\tau,X}^2\leq K_4.$$

	For such a pair of parameters $(t,\tau)$, let $\{[A_n,B_n]\}_{n\in\mathbb{N^*}}\subseteq\mathcal{M}^*_{PVRVW,0}(t,\tau)$ be a sequence of the solutions of \eqref{yz}. Theorem \ref{tk} implies that we can assume $A_n$ and $B_n$ are all $C^\infty$ for every $n\in\mathbb{N^*}$, and the choice of $(t,\tau)$ and Theorem \ref{pb} allow us to further assume $d^*A_n=0$ and $$\vert\vert A_n\vert\vert_{L_k^2}\leq C_4\vert\vert F_{A_n}\vert\vert_{L^2}\leq C_4K_4.$$
	Also, for some constant $c$,
	\begin{align*}
		\vert\vert d_{A_n}d_{A_n}B_n\vert\vert_{L^2}&=\vert\vert [F^+_{A_n},B_n]\vert\vert_{L^2}\\
		&\leq c\vert\vert F^+_{A_n}\vert\vert_{L^2}\vert\vert B_n\vert\vert_{L^\infty}\\
		&\leq c\sqrt{\frac{\mathrm{vol}(X)}{6}}C_{t,\tau,X}^2\cdot C_{t,\tau,X}\\
		&=c\sqrt{\frac{\mathrm{vol}(X)}{6}}C_{t,\tau,X}^3,\\
		\vert\vert d_{A_n}^*d_{A_n}d_{A_n}B_n\vert\vert_{L^2}&=\vert\vert d_{A_n}^*[F^+_{A_n},B_n]\vert\vert_{L^2}\\
		&=\frac{1}{8}\vert\vert d_{A_n}^*[[B_n\centerdot B_n],B_n]\vert\vert_{L^2}\\
		&\leq c\vert\vert d_{A_n}^*B_n\vert\vert_{L^2}\vert\vert B_n\vert\vert_{L^\infty}^2\\
		&\leq cK_{t,\tau,X}C_{t,\tau,X}^2.
	\end{align*}
	So inductively, the bounds $\vert\vert B\vert\vert_{L^\infty} \leq C_{t,\tau,X}$  and $\vert\vert d_A^*B\vert\vert_{L^2}=\vert\vert d_AB\vert\vert_{L^2}\leq K_{t,\tau,X}$ imply that there is a constant $K'_{t,\tau,X}>0$, depends only on $t,\tau$ and $X$, such that
	$$\vert\vert B_n\vert\vert_{L_k^2}\leq K'_{t,\tau,X}.$$
	Note that for $k\geq 3$,  Rellich's embedding theorem asserts that the embedding $L_k^2\subset C^{k-2}$ is compact, hence $\{(A_n,B_n)\}_{n\in\mathbb{N^*}}$ contains a convergent subsequence $\{(A_{n_l},B_{n_l})\}_{l\in\mathbb{N^*}}$, which is $L_k^2$ convergent to a limit $(A',B')\in L_{k}^2(X,\mathfrak{su}(2)_P\otimes\Lambda^{1})\oplus L_{k}^2(X,\mathfrak{su}(2)_P\otimes\Lambda^{2,+})$, and it's easy to see that $(A',B')$ is also a solution to \eqref{yz}. 
	
	\section*{Appendix}
	\appendix
	\setcounter{theorem}{0}
	\renewcommand\thetheorem{A.\arabic{theorem}}
	
	In this appendix, we prove the Lemma used in section \ref{trans}.

	\begin{lemma}\label{xx} Suppose that $B\in\mathfrak{su}(2)\otimes\Lambda^{2,+}\mathbb{R}^4$ has rank 3. Then the linear map
		$$\begin{aligned}
			M_B:\  \mathfrak{su}(2)\otimes\Lambda^{2,+}\mathbb{R}^4&\longrightarrow\mathfrak{su}(2)\otimes\Lambda^{2,+}\mathbb{R}^4 ,\\
			\psi&\longmapsto[B\centerdot\psi]\\
		\end{aligned}$$
		is injective, hence it is an isomorphism.
	\end{lemma}
	
	\begin{proof} As in the proof of Lemma~\ref{xzz}, there exist an oriented orthonormal basis $\{e^1,e^2,e^3,e^4\}$ for $\mathbb{R}^4$, and $\{\eta_1,\eta_2,\eta_3\}$  for the  Lie algebra $\mathfrak{su}(2)$ such that $B=B_{11}\eta_1(e^{12}+e^{34})+B_{22}\eta_2(e^{13}+e^{42})+B_{33}\eta_3(e^{14}+e^{23})$, where $B_{11},B_{22},B_{33}\in\R$.  Recall that $[\eta_1,\eta_2] =2\eta_3$ and cyclic permutations. 
		
		Since $\rank(B)=3$, we have $B_{11}B_{22}B_{33}\neq 0$. Let $\psi\in\mathfrak{su}(2)\otimes\Lambda^{2,+}\mathbb{R}^4$, and assume
		$$\begin{aligned}
			\psi&=(\psi_{11}\eta_1+\psi_{12}\eta_2+\psi_{13}\eta_3)(e^{12}+e^{34})\\
			&+(\psi_{21}\eta_1+\psi_{22}\eta_2+\psi_{23}\eta_3)(e^{13}+e^{42})\\
			&+(\psi_{31}\eta_1+\psi_{32}\eta_2+\psi_{33}\eta_3)(e^{14}+e^{23}),
		\end{aligned}$$
		where $\psi_{ij}\in\R$, $i,j=1,2,3$.
		
		If $[B\centerdot\psi]=0$, then we have
		\begin{align*}
			0=[B\centerdot\psi]&=-2\big([B_{11}\eta_1,\psi_{21}\eta_1+\psi_{22}\eta_2+\psi_{23}\eta_3]\\
			&-[B_{22}\eta_2,\psi_{11}\eta_1+\psi_{12}\eta_2+\psi_{13}\eta_3]\big)(e^{14}+e^{23})\\
			&-2\big([B_{22}\eta_2,\psi_{31}\eta_1+\psi_{32}\eta_2+\psi_{33}\eta_3]\\
			&-[B_{33}\eta_3,\psi_{21}\eta_1+\psi_{22}\eta_2+\psi_{23}\eta_3]\big)(e^{12}+e^{34})\\
			&-2\big([B_{33}\eta_3,\psi_{11}\eta_1+\psi_{12}\eta_2+\psi_{13}\eta_3]\\
			&-[B_{11}\eta_1,\psi_{31}\eta_1+\psi_{32}\eta_2+\psi_{33}\eta_3]\big)(e^{13}+e^{42})\\
			&=-4(B_{11}\psi_{22}\eta_3-B_{11}\psi_{23}\eta_2+B_{22}\psi_{11}\eta_3-B_{22}\psi_{13}\eta_1)(e^{14}+e^{23})\\
			&-4(-B_{22}\psi_{31}\eta_3+B_{22}\psi_{33}\eta_1-B_{33}\psi_{21}\eta_2+B_{33}\psi_{22}\eta_1)(e^{12}+e^{34})\\
			&-4(B_{33}\psi_{11}\eta_2-B_{33}\psi_{12}\eta_1-B_{11}\psi_{32}\eta_3+B_{11}\psi_{33}\eta_2)(e^{12}+e^{34}).\
		\end{align*}
		So,
		$$\left\{
		\begin{aligned}
			&B_{11}\psi_{22}+B_{22}\psi_{11}=B_{11}\psi_{23}=B_{22}\psi_{13}=0,\\
			&B_{22}\psi_{33}+B_{33}\psi_{22}=B_{22}\psi_{31}=B_{33}\psi_{21}=0,\\
			&B_{33}\psi_{11}+B_{11}\psi_{33}=B_{33}\psi_{12}=B_{11}\psi_{32}=0.\\
		\end{aligned}\right.$$
		Note that, since $B_{11}B_{22}B_{33}\neq 0$, we have $\psi_{ij}=0$ for $i\neq j$. For $\psi_{11},\psi_{22},\psi_{33}$, the determinant
		$$\begin{vmatrix}
			B_{22} & B_{11} & 0 \\
			0 & B_{33} & B_{22} \\
			B_{33} & 0 & B_{11} \\
		\end{vmatrix}=2B_{11}B_{22}B_{33}\neq 0,$$
		so, $\psi_{11}=\psi_{22}=\psi_{33}=0$, and we have $\psi=0$. Hence $M_B$ is injective, and also, it is an isomorphism.
	\end{proof}

\end{document}